\documentclass[12pt,a4paper]{article}
\usepackage{amsmath,amssymb}

\newtheorem{theorem}{Theorem}[section]

\newtheorem{lemma}{Lemma}[section]
\newtheorem{corollary}{Corollary}[section]

\newenvironment{proof}{\noindent Proof:}{$\Box$}

\newcommand{\Z}{{\mathbb Z}}
\newcommand{\Q}{{\mathbb Q}}
\newcommand{\N}{{\mathbb N}}

\newcommand{\mult}{{\mathrm {mult}\,}}
\newcommand{\length}{{\mathrm {length}\,}}
\newcommand{\codim}{{\mathrm {codim}\,}}
\newcommand{\Asc}{{\mathcal A}}

\title{
Length and multiplicity of the local cohomology 
with support in a hyperplane arrangement
}
\author{Toshinori Oaku} 
\date{September 4, 2015}

\begin{document}
\maketitle

\begin{abstract}
Let $R$ be the polynomial ring 
in $n$ variables with coefficients in a field $K$ 
of characteristic zero. 
Let $D_n$ be the $n$-th Weyl algebra over $K$. 
Suppose that  $f \in R$ defines a hyperplane arrangement in the affine space $K^n$. 
Then the length and the multiplicity of the first local cohomology group $H^1_{(f)}(R)$ 
as left $D_n$-module coincide and are explicitly expressed in terms of 
the Poincar\'e polynomial or the M\"obius function of the arrangement. 
\end{abstract}

\section{Introduction}

Let $K$ be a field of characteristic zero 
and $R = K[x] = K[x_1,\dots,x_n]$ be the polynomial ring 
in $n$ variables $x = (x_1,\dots,x_n)$. 
For a nonzero polynomial $f \in K[x]$, let us consider 
the first local cohomology group $H^1_{(f)}(R) = R[f^{-1}]/R$, 
where $R[f^{-1}] = R_f$ is the localization of $R$ with respect to 
the multiplicative set $\{f^k \mid k\in \N\}$ with 
$\N = \{0,1,2,3,\dots\}$.  

Let $D_n = R\langle \partial\rangle 
= R\langle \partial_{1},\dots, \partial_{n} \rangle $ 
be the $n$-th Weyl algebra, i.e., the ring of differential 
operators with polynomial coefficients with respect to 
the variables $x$, where we denote 
$\partial = (\partial_{1},\dots,\partial_{x})$ 
with $\partial_{i} = \partial/\partial x_i$ being the
derivation with respect to $x_i$. 
An arbitrary element $P$ of $D_n$ is written in a finite sum
\[
P = \sum_{\alpha\in\N^n} a_\alpha(x)\partial^\alpha 
\quad \mbox{with}\quad
a_\alpha(x) \in K[x], 
\]
where we denote 
$\partial^\alpha 
= \partial_{1}^{\alpha_1}\cdots\partial_{n}^{\alpha_n}$ 
for a multi-index $\alpha = (\alpha_1,\dots,\alpha_n) \in \N^n$.

Then $H_{(f)}^1(R)$ has a natural structure of left $D_n$-module 
and is holonomic (\cite{K}).  
We are mainly concerned with the length and the multiplicity 
of $H_{(f)}^1(R)$ as a left $D_n$-module in case $f$ defines a 
hyperplane arrangement $\Asc$ in the affine space $K^n$; i.e., 
$f$ is a multiple of linear (i.e., first-degree) polynomials. 
In particular, we show that the length and the multiplicity both 
coincide with $\pi(\Asc,1)-1$, where $\pi(\Asc,t)$ is the 
Poincar\'e polynomial of the arrangement $\Asc$.

The length of $R[f^{-1}]$ as left $D_n$-module, 
which equals that of $H^1_{(f)}(R)$ plus one,  
with $f$ defining a hyperplane arrangement was studied 
e.g., in \cite{AB}, \cite{W}. 
The characteristic cycle of the local cohomology with respect to 
an arrangement of linear subvarieties was studied in \cite{AGZ}. 
Although not explicitly stated, 
Corollary 1.3 of \cite{AGZ} should yield main results of this paper. 
We give a direct proof for hyperplane arrangements.

\section{Length and multiplicity}

First let us recall basic facts about the length and the multiplicity 
of a left $D_n$-module following J. Bernstein (\cite{B}). 
Let $M$ be a finitely generated left $D_n$-module. 
A composition series of $M$ of length $k$ is a sequence 
\[
M = M_0 \supset M_1 \supset \cdots \supset M_k = 0
\]
of left $D_n$-submodules such that 
$M_i/M_{i-1}$ is a nonzero simple left $D_n$-module 
(i.e.\ having no proper left $D_n$-submodule other than $0$)  
for $i=1,\dots k$. 
The length of $M$, which we denote by $\length M$, is 
the least length of composition series (if any) of $M$. 
If there is no composition series, the length of $M$ is defined 
to be infinite.  
The length is additive in the sense that if 
\[
 0 \longrightarrow N \longrightarrow M \longrightarrow L \longrightarrow 0
\]
is an exact sequence of left $D_n$-modules of finite length, 
then $\mult M = \mult N + \mult L$ holds. 

For each integer $k$, set
\[
F_k(D_n) = \{\sum_{|\alpha| + |\beta| \leq k}
a_{\alpha\beta}x^\alpha\partial^\beta \mid a_{\alpha\beta}\in K\}. 
\]
In particular, we have $F_k(D_n) = 0$ for $k < 0$ and $F_0(D_n) = K$. 
The filtration $\{F_k(D_n)\}_{k\in\Z}$ 
is called the Bernstein filtration on $D_n$. 

Let $M$ be a finitely generated left $D_n$-module. 
A family $\{F_k(M)\}_{k\in \Z}$ 
of $K$-subspaces of $M$ is called a Bernstein filtration of $M$ if it 
satisfies
\begin{enumerate}
\item 
$F_k(M) \subset F_{k+1}(M)\quad (\forall k\in \Z), 
\qquad \bigcup_{k \in \Z} F_k(M) = M$
\item
$F_j(D_n)F_k(M) \subset F_{j+k}(M)\quad (\forall j,k\in \Z)$
\item
$F_k(M) = 0$ for $k \ll 0$
\end{enumerate} 

Moreover, $\{F_k(M)\}$ is called a good Bernstein filtration if
\begin{enumerate}
\setcounter{enumi}{3}
\item
$F_k(M)$ is finite dimensional over $K$ for any $k \in \Z$.
\item
$F_j(D_n)F_k(M) = F_{j+k}(M)$ $(\forall j \geq 0)$ holds for $k \gg 0$. 
\end{enumerate}

Let $\{F_k(M)\}$ be a good Bernstein filtration of $M$. 
Then there exists a polynomial  
$h(T) = h_dT^d + h_{d-1}T^{d-1} + \cdots + h_0\in \Q[T]$ such that 
\[
\dim_K F_k(M) = h(k) \quad (k \gg 0)
\]
and $d!h_d$ is a positive integer.  
We call $h(T)$ the Hilbert polynomial of $M$ with respect to 
the filtration $\{F_k(M)\}$. 
The leading term of $h(T)$ does not depend on the choice of 
a good Bernstein filtration $\{F_k(M)\}$.  
The degree $d$ of the Hilbert polynomial $h(T)$ is 
called the dimension of $M$ and denoted by $\dim M$. 
The multiplicity of $M$ is defined to be $d!h_d$, which 
we denote by $\mult M$. 

If $M \neq 0$, then the dimension of $M$ is not less than $n$ 
(Bernstein's inequality). 
We call $M$ holonomic if $M=0$ or $\dim M = n$. 
It is known that $H^j_I(R)$ is holomomic for any ideal $I$ of $R$ 
and any integer $j$ (\cite{K}). 

If $M$ is a holonomic left $D_n$-module, we have an equality
$\length M \leq \mult M$. Moreover, the multiplicity 
is additive for holonomic left $D_n$-modules.

\begin{lemma}\label{lemma:1}
Let $h_0 = h_0(x) \in K[x]$ be a linear polynomial  
and $I$ be an ideal of $R := K[x]$. 
Let $R':= R/Rh_0$ be the affine ring 
associated with the hyperplane $h_0(x) = 0$ 
and set $I' = (I + Rh_0)/Rh_0$. 
Then we have 
\[
\length H_{I + Rh_0}^i(R) = \length H_{I'}^{i-1}(R'),
\qquad
\mult H_{I + Rh_0}^i(R) = \mult H_{I'}^{i-1}(R')
\]
for any integer $i$. 
\end{lemma}

\begin{proof}
Since $H^i_{Rh_0}(R) = 0$ for $i \neq 1$, there is an isomorphism
\[
H^i_{I+Rh_0}(R) \cong H^{i-1}_{I+Rh_0}(H^1_{Rh_0}(R)).
\]
We may assume by an affine coordinate transformation, 
which preserves the Bernstein filtration, that 
$h_0(x) = x_n$. 
Then we may regard $R' = K[x_1,\dots,x_{n-1}]$ and 
have an isomorphism
\[
H^1_{Rh_0}(R) \cong R'\otimes_K H^1_{(x_n)}(K[x_n]) 
, 
\]
where the tensor product on the right-hand side is a left module
over $D_n = D_{n-1}\otimes_K D_1$ with $D_1$ being the ring of differential 
operators in the variable $x_n$. 

Let $\{f_1,\dots,f_r\}$ be a set of generators of $I$. 
We may assume 
that $f_1,\dots,f_r$ belong to $R'$. 
Then for $0 \leq i_1 < \cdots < i_k \leq r$ with $k \in \N$, 
the localization by $f_{i_1}\cdots f_{i_k}$ yields
\[
(H^1_{Rh_0}(R))_{f_{i_1}\cdots f_{i_k}}
= R'_{f_{i_1}\cdots f_{i_k}}\otimes_K H^1_{(x_n)}(K[x_n]). 
\] 
On the other hand, we have 
\[
(H^1_{Rh_0}(R))_{x_n} = R'\otimes_K (H^1_{(x_n)}(K[x_n]))_{x_n}
= 0. 
\]
Hence $H^{i-1}_{I+Rh_0}(H^1_{Rh_0}(R))$ is the $(i-1)$-th cohomology group 
of the \v{C}ech complex
\begin{multline}
0 \longrightarrow R'\otimes_K H^1_{(x_n)}(K[x_n]) 
\longrightarrow \bigoplus_{1 \leq i \leq r}R'_{f_i}\otimes_K H^1_{(x_n)}(K[x_n]) 
\\
\longrightarrow \bigoplus_{1 \leq i_1 < i_2\leq r}R'_{f_{i_1}f_{i_2}}\otimes_K H^1_{(x_n)}(K[x_n]) 
\longrightarrow \cdots \longrightarrow 
R'_{f_{1}\cdots f_{r}}\otimes_K H^1_{(x_n)}(K[x_n]) \longrightarrow 0, 
\nonumber
\end{multline}
which is isomorphic to $H^{i-1}_{I'}(R') \otimes_K H^1_{(x_n)}(K[x_n])$. 
This implies
\begin{align*}
H^i_{I+Rh_0}(R) &\cong H^{i-1}_{I'}(R')\otimes_K H^1_{(x_n)}(K[x_n])
\\&
\cong  H^{i-1}_{I'}(R')\otimes_K (D_1/D_1x_n)
\cong  (D_n/D_nx_n)\otimes_{D_{n-1}} H^{i-1}_{I'}(R')
,
\end{align*}
where $D_n/D_{n}x_n$ is regarded as a $(D_n,D_{n-1})$-bimodule. 
The rightmost term is the $D$-module theoretic direct image 
of $H^{i-1}_{I'}(R')$ 
with respect to 
the inclusion $H_0 := \{x \in V \mid x_n=0\} \rightarrow V$. 
In view of Kashiwara's equivalence in the category of algebraic $D$-modules 
(see e.g., Theorem 7.11 of \cite{Borel} or Theorem 1.6.1 of \cite{HTT}), 
there is a one-to-one correspondence 
between the $D_{n-1}$-submodules $M$ of $H^{i-1}_{I'}(R')$ and 
the $D_n$-submodules $M \otimes_K H^1_{(x_n)}(K[x_n])$ 
of $H^i_{I+Rh_0}(R)$.  
This implies 
\[
\length H^i_{I+Rh_0}(R) = \length H^{i-1}_{I'}(R'). 
\]

Next, let us show 
\[
\mult H^i_{I+Rh_0}(R) = \mult H^{i-1}_{I'}(R').
\]
Let $\{F_k\}$ be a good Bernstein filtration of $H^{i-1}_{I'}(R')$ 
and set $m = \mult H^{i-1}_{I'}(R')$.
We may assume $F_k = 0$ for $k <0$. 
Then there exists a polynomial $p(k)$ and $k_1 \in \Z$ such that
\[
 \dim_K F_k = p(k) = \frac{m}{(n-1)!}k^{n-1} 
+ \mbox{(terms with degree $< n-1$)}
\]
holds for $k \geq k_1$. Define a filtration $\{G_k\}$ 
on $H^{i-1}_{I'}(R')\otimes_K H^1_{(x_n)}(K[x_n^{-1}])$ by
\[
G_k := \sum_{j=0}^k F_j\otimes_K (K[x_n^{-1}] + \cdots 
+ K[x_n^{-(k-j)-1}])
= \bigoplus_{j=0}^k F_j \otimes_K K[x_n^{-(k-j)-1}]
. 
\]
It is easy to see that $\{G_k\}$ is a good Bernstein filtration. 
Hence we have
\[
\dim_K G_k = \sum_{j=0}^k \dim_K F_j 
= \sum_{j=0}^{k_1-1} \dim_K F_j 
+ \sum_{j=k_1}^{k} p(j).
\]
By the assumption, there exists a polynomial $q(k)$ of degree 
$\leq n-2$ such that 
\[
p(j) = \frac{m}{(n-1)!}j(j+1)\cdots (j+n-2) + q(j). 
\]
Since
\[
\sum_{j=k_1}^{k}j(j+1)\cdots (j+n-2) 
= \frac{1}{n}\{k(k+1)\cdots (k+n-1) - (k_1-1)k_1\cdots (k_1+n-2)\}
,
\]
we have
\[
\dim_K G_k = \frac{m}{n!}k^n 
+ \mbox{(terms with degree $< n$)} 
\qquad (\forall k \geq k_1).
\]
Thus we also have $\mult H^i_{I + Rx_n}(R) = m$. 
This completes the proof. 
\end{proof}

\section{Hyperplane arrangements}

Let $f \in K[x]$ be a multiple of essentially distinct 
linear polynomials. 
Let $\Asc$ be the hyperplane arrangement in $V := K^n$ defined by 
$f$. 

\begin{theorem}\label{th:th1}
Let $H_0$ be an element  of $\Asc$. 
Set $\Asc':= \Asc \setminus \{H_0\}$ 
and let $f'$ be the product of the defining polynomials 
of hyperplanes belonging to $\Asc'$.  
Let us regard
\[
\Asc'' := \{H \cap H_0 \mid H \in \Asc', \, 
H \cap H_0 \neq \emptyset\}
\]
as a hyperplane arrangement in the affine space $H_0$.  
Let $R'' = R/R h_0$ be the affine ring of $H_0$, where 
$h_0$ is a polynomial of first degree defining $H_0$.
Let $f'' \in R''$ be the product of the defining polynomials of 
the elements of $\Asc''$.
Then we have
\begin{align*}&
\length H_{(f)}^1(R) = \length H_{(f')}^1(R) + \length H^1_{(f'')}(R'') + 1,
\\&
\mult H_{(f)}^1(R) = \mult H_{(f')}^1(R) + \mult H^1_{(f'')}(R'') + 1.
\end{align*}
\end{theorem}

\begin{proof} 
By the  Mayer-Vietoris exact sequence, we get an exact sequence 
\[
0 \longrightarrow  H^1_{(f')}(R) \oplus H^1_{(h_0)}(R) 
\longrightarrow
H^1_{(f)}(R) \longrightarrow H^2_{(f')+(h_0)}(R) 
\longrightarrow 0
\]
of holonomic left $D_n$-modules. 
Since the length and the multiplicity of $H^1_{(h_0)}(R)$ are both one, 
it follows that
\begin{align}
\length H_{(f)}^1(R) &= \length H_{(f')}^1(R) + \length H^2_{(f') +(h_0)}(R) + 1,
\nonumber\\
\mult H_{(f)}^1(R) &= \mult H_{(f')}^1(R) + \mult H^2_{(f')+(h_0)}(R) + 1.
\label{eq:MV}
\end{align}
Since $(f'') = R''f'' \cong (Rf'+Rh_0)/Rh_0$, 
Lemma \ref{lemma:1} implies 
\begin{align*}
\mult H^2_{(f')+(h_0)}(R) &= \mult H^1_{(f'')}(R''), 
\\
\length H^2_{(f')+(h_0)}(R)& = \length H^1_{(f'')}(R'').   
\end{align*}
This completes the proof in view of (\ref{eq:MV}).
\end{proof}

\begin{corollary}
$\length H_{(f)}^1(R) = \mult H_{(f)}^1(R)$. 
\end{corollary}

\begin{proof}
This can be easily proved by induction on $\sharp\Asc$ 
by using Theorem \ref{th:th1}. 
\end{proof}

The intersection poset  $L(\Asc)$ is the set of 
the non-empty intersections of elements of $\Asc$. 
For $X,Y \in L(\Asc)$, define $X \leq Y$ if and only if 
$X \supset Y$. 
For $X,Y \in L(\Asc)$, 
the M\"obius function $\mu(X,Y)$ is defined recursively by
\[
\mu(X,Y) = 
\left\{ \begin{array}{ll}\displaystyle
-\sum_{X \leq Z < Y} \mu(X,Z) & 
\mbox{if $X < Y$ } 
\\
1 & \mbox{if $X = Y$} 
\\
0 & \mbox{otherwise}
\end{array}\right.
\]
Set $\mu(X) = \mu(V,X)$. 
Then $(-1)^{\codim X}\mu(X)$ is positive 
(see e.g.\ Theorem 2.47 of \cite{OT}). 
The Poincar\'e polynomial of the arrangement $\Asc$ is defined by
\[
\pi(\Asc,t) = \sum_{X \in L(\Asc)} \mu(X)(-t)^{\codim X}.
\]

\begin{theorem}\rm
\[
\length H_{(f)}^1(R)  = \pi(\Asc, 1) -1 
= \sum_{X \in L(\Asc) \setminus\{V\}} |\mu(X)|. 
\]
\end{theorem}

\begin{proof}
Let $H_0$ be an element of $\Asc$ defined by a first degree polynomial $h_0$. 
Let us prove the equality by induction on $\sharp \Asc$. 
Since $H_{(h_0)}(R)$ is simple, the equality holds if $\Asc = \{H_0\}$ 
with $\pi(\Asc,1) = 2$.  
Let $\Asc'$, $\Asc''$ be as in the proof of Theorem \ref{th:th1}. 
By the induction hypothesis, we have
\[
\length H^1_{(f')}(R) = \pi(\Asc',1)-1, 
\qquad
 \length H^1_{(f'')}(R'') = \pi(\Asc'',1) -1.
\]
Hence by Theorem \ref{th:th1} we get
\begin{align*}
\length H^1_{(f)}(R) &= \length H^1_{(f')}(R) + \length H^1_{(f'')}(R'') + 1 
\\&
= \pi(\Asc',1) + \pi(\Asc'',1) - 1.
\end{align*}
On the other hand, $\pi(\Asc,t) = \pi(\Asc',t) + t\pi(\Asc'',t)$ 
holds (see e.g., Theorem 2.56 of \cite{OT}). 
Thus we get
\begin{align*}
\length H_{(f)}^1(R) &= \pi(\Asc',1) + \pi(\Asc'',1) -1 
= \pi(\Asc,1) -1. 
\end{align*}
This completes the proof. 
\end{proof}

\end{document}